\documentclass[reqno,b5paper]{amsart}
\usepackage{amsmath}
\usepackage{amssymb}
\usepackage{amsthm}
\usepackage[nopar]{lipsum}
\usepackage{enumerate}
\usepackage[mathscr]{eucal}
\usepackage{adjustbox}
\usepackage{float}
\usepackage{longtable}
\usepackage{lipsum}
\usepackage{amssymb}

\def\mod#1{{\ifmmode\text{\rm\ (mod~$#1$)}
\else\discretionary{}{}{\hbox{ }}\rm(mod~$#1$)\fi}}

\usepackage{float}
\restylefloat{table}

\newtheorem{theorem}{Theorem}[section]
\newtheorem{lemma}[theorem]{Lemma}
\newtheorem{corollary}[theorem]{Corollary}

\theoremstyle{remark}

\numberwithin{equation}{section}

\def\mod#1{{\ifmmode\text{\rm\ (mod~$#1$)}
\else\discretionary{}{}{\hbox{ }}\rm(mod~$#1$)\fi}}

\usepackage{hyperref}

\begin{document}
\title{On the factorization of $x^2+D$}

\author{Amir Ghadermarzi}
\address{School of Mathematics, Statistics and Computer Science, College of Science, University of Tehran, Tehran, Iran\\
 \&  School of Mathematics, Institute of Research in Fundamental Science (IPM), P.O.BOX 19395-5746  }
\email{a.ghadermarzi@ut.ac.ir}
\thanks{This research was in part supported by a grant from IPM(No.95110044)}


\begin{abstract}
Let $D$ be a positive nonsquare integer, $p$ a prime number with $p \nmid D$, and $0< \sigma < 0.847$. We show that if the equation $x^2+D=p^n$ has a huge solution $(x_0,n_0)_{(p,\sigma)}$, then there exists an effectively computable  constant $C_p$ such that for every $x> C_P$ with $x^2+D=p^n.m $, we have $ m > x^{\sigma}$. As an application, we show that for $x \neq \{1015,5 \}$, if the equation $x^2+76=101^n.m $ holds, we have $ m > x^{0.14}$. 
\end{abstract}

\maketitle

\section{Introduction} \label{intro section}
Let $f(x) $ be a polynomial with integer coefficients. The prime factorization of $f(n)$ is an important question in Diophantine approximation. One direction in considering this question is the following: let $p_1, p_2,\cdots,p_s$ be a set of prime numbers and $f(n)=p_1^{\alpha_1}p_2^{\alpha_2} \cdots p_s^{\alpha_s} . m$ where $(m ,\prod p_i)=1 $, then find a bound on size of $m$ in terms of $n$. From Mahler's work \cite{M} we know that for any $\epsilon > 0$ there exists a constant $C$ such that for any $n$ bigger than $C$, we have $m > n^{1-\epsilon}$. Mahler's proof depends on the p-adic version of Roth's theorem. Thus, One can not effectively compute the value $C$ using Mahler's result. While quantifying this result in general is a difficult task, there are some effective results. Stewart \cite{S} proved effective results for product of consecetive integers using linear form in logarithms. He proved that if $n(n+1) \cdots(n+k)= p_1^{\alpha_1}p_2^{\alpha_2} \cdots p_s^{\alpha_s} . m$ where $(m ,\prod p_i)=1 $, then  $m \gg x^{\sigma(p_1,p_2,\cdots p_s)}$ for some small number $\sigma(p_1,p_2,\cdots p_s) >0.$ A result of the same flavor is obtained by Gross and Vincent \cite{G}. They generalized Stewart's result for polynomials with at least two distinct roots. 
In some specific cases, using Pad\'{e} approximations, better results were obtained by Bennett, Filaseta and Trifonov. 
 They proved that if $n^2+n=2^{\alpha}3^{\beta}m$ and $n>8 $ then $m > n^{0.285}$ \cite{BFT2}, they also showed that if $n^2+7=2^{\alpha}.m$ then either $n \in \{1,3,5,11,181 \} $ or $m > n^{\frac{1}{2}} $ \cite{BFT1}. In this note, we generalize the former result for generalized Ramanujan-Nagell type equations. A generalization of Ramanujan-Nagell equation $ x^2+7=y^n $ of the form
\begin{equation} \label{genram}
x^2+D=\lambda y^n
\end{equation}
is called generalized Ramanujan-Nagell equation. There is a vast literature on bounding the number of solutions for general values of $D$ and explicit determination of the exact solutions for specific values of $D$ in the equation \eqref{genram} (see \cite{Ba}, \cite{Beu1}, \cite{Beu2}, \cite{Bu1} and \cite{Bu2}).
 In this note, we consider a generalization of Ramanujan-Nagell equation of the form $x^2+D$ towards its factorization.
 From now on, let $D$ be a  nonsquare positive integer bigger than 12, and let $p$ be a prime number with $p \nmid D$. Our result relies on the existence of a relatively huge solution to the equation $x^2+D=p^n$. In order to get effective quantitative results, we explicitly determine what we mean by a huge solution. Obviously, the bigger the solution is, the easier it is to obtain a better result. We set a condition for a huge solution that obviously depends on $\sigma$. Assume $(x_0,n_0)$ is a positive solution to the equation $x^2+D=p^n$ with $p \nmid D$, that satisfies the condition:\\
 \begin{equation} \label{sigmacondition}
  \begin{split}
    \left \lvert  x_0+\sqrt{-D} \right \rvert &=  p^{\frac{n_0}{2}} >  C_{\sigma,p}  D^{\eta} \quad \text{when $p$ is an odd prime, and} \\
  \frac{\left \lvert  x_0+\sqrt{-D} \right \rvert}{2} &= 2^{\frac{n_0}{2}-1}>  C_{\sigma,2}    D^{\eta}\quad  \text{when $p$ equals to 2}, 
  \end{split}
  \end{equation}
where $ \eta_{\sigma}=  \frac{7.84-4\sigma}{7.64-9\sigma} $, $ C_{\sigma,p}=(2008.832)^{\frac{1.96-\sigma}{7.64-9\sigma}},$ if $p$ is an odd prime 
 and $C_{\sigma,2}= (7.847)^{\frac{1.96-\sigma}{7.64-9\sigma}}$.\\
We call such a solution a huge solution. For such solutions from now on, let's define:
$$  \beta=  x_0+\sqrt{-D},   \text{  with  } \left\lvert \beta \right \rvert =  p^{\frac{n_0}{2}}   \quad \text{when $p$ is an odd prime, and}$$ 
$$ \beta=  \frac{  x_0+\sqrt{-D} }{2},  \text{  with  } \left\lvert \beta \right \rvert= 2^{\frac{n_0}{2}-1}  \quad \text{when $p$ equals to 2 }. $$ 

   \begin{theorem} \label{theorem1}
Assume that the equation $x^2+D=p^n$ has a solution which satisfies the condition of \eqref{sigmacondition}. Let $M=250n_0$, and $\left \lvert \beta \right \rvert  \geq 90.93$. If $x^2+D=p^n.m $, with $n >M$, then $m > x^{\sigma}$.  
  \end{theorem}
  Since $\sigma < 1$, as an immediate corollary one gets the following:
  \begin{corollary} \label{corollary1}
 If $x > p^{250n_0}$ and $x^2+D=p^n.m$, then $m > x^{\sigma}$.
  \end{corollary}

 The idea of the proof is to approximate $\left(\frac{\bar{\beta}}{\beta}\right)^k$ with rational and algebraic numbers. In the equation $x^2+D=p^n.m$, whenever $m$ is small compared to $ \left \lvert \beta \right \rvert $, we get a good approximation for  $\left(\frac{\bar{\beta}}{\beta}\right)^k$. If the equation $x^2+D=p^n$ has a huge solution, then there exists a good approximation for $k=1$. We use Pad\'{e} approximation method to show that if there exists a huge solution we have a good approximation for $k=1$. However we are unable to get a good approximation ($m$ small) for larger values of $k$. 
 
 We proceed as follows: in section \ref{section2}, we use Pad\'{e} approximations to the function $(1-x)^k$ to generate approximations for $\left( \frac{\bar{\beta}}{\beta} \right)^k$. In section \ref{section3}, using a huge solution of $x^2+D=y^n$, we get an approximation for $\left( \frac{\bar{\beta}}{\beta} \right)^k$. In section \ref{section4}, by combining approximations of sections \ref{section2} and \ref{section3}, we prove Theorem \ref{theorem1}. In section \ref{finalsection} we  consider the equation $x^2+78=101^n.m$ to show that the value $M$ in theorem \ref{theorem1} is actually not huge, and the equation $x^2+D=p^n.m$ can be solved relatively fast  for any $x$ less than $p^M$.

\section{Pad\'{e} approximation} \label{section2}
In this section, we use Pad\'{e} approximations to the function $(1-x)^k$ to get approximations of $ \left(\frac{\bar{\beta}}{{\beta}}\right )^k$. We follow Bennett \cite{B} to produce these approximations.
Let $A$, $B$ and $C$ be positive integers. Define :
\begin{equation} \label{pade polynomials}
\begin{split}
 P_A(z)= \frac{(A+B+C+1)!}{A! B! C!}& \int_{0}^{1} t^A \left(1-t \right)^B (z-t)^C dt, \\
 Q_A(z)=  \frac{(-1)^C(A+B+C+1)!}{A! B! C!}& \int_{0}^{1} t^B \left(1-t \right)^C (1-t+zt)^A dt, \\
 E_A(z)= \frac{(A+B+C+1)!}{A! B! C!}& \int_{0}^{1} t^A \left(1-t \right)^C (1-zt)^B dt .
\end{split}   
\end{equation}
From these definitions we have 
\begin{equation} \label{nonzero}
P_A(z)-(1-z)^{B+C+1} Q_A(z)=z^{A+C+1} E_A(z).
\end{equation}
So these equations provide a set of approximations to $(1-z)^{B+C+1}$.
Expanding the formulas above, we have the following lemma, which is lemma 1 in \cite{B}.
\begin{lemma} \label{second form}
With the same conditions as above, the functions $P_A$, $Q_A$ and $E_A$ satisfy :
\begin{equation}
\begin{split}
 P_A(z)= \sum_{i=0}^{C} \binom{A+B+C+1}{i} \binom{A+C-i}{A}(-z)^i,\\
 Q_A(z)= (-1)^{C}\sum_{i=0}^{A} \binom{A+C-i}{C} \binom{B+i}{i}(z)^i,\\
 E_A(z)= \sum_{i=0}^{B} \binom{A+i}{i} \binom{A+B+C+1}{A+C+i+1}(-z)^i.
\end{split} 
 \end{equation}
\end{lemma}Moreover, we have $P_A(z)Q_{A+1}(z)-Q_A(z)P_{A+1}(z)= cz^{A+C+1}$ \cite{B}, which means that we can get distinct approximations for $(1-z)^{B+C+1}$.
For our purpose, namely for finding approximations to $(\frac{\bar{\beta}}{\beta})^k$, we take $A=C=r$ and $B=k-r-1$. Note that by taking $A=C$, we get a diagonal approximation and from our choice of $B$, we get the desired approximation of $(1-z)^k$. With this choice of parameters, equations \eqref{second form} can be rewriten as the following lemma (see \cite{BFT1}): 
\begin{lemma} \label{lemma2.1}
For positive integers $k$ and $r$ with $k>r$, there exist polynomials $P_r(z)$, $Q_r(z)$ and $E_r(z)$ with integer coefficients such that 
\begin{enumerate}
\item $Q_r(z)=  \frac{(k+r)!}{(k-r-1)! r! r!} \int_{0}^{1} t^{k-r-1} \left(1-t \right)^r (1-t+zt)^r dt, $\\

\item $E_r(z)=  \frac{(k+r)!}{(k-r-1)! r! r!} \int_{0}^{1} t^{r} \left(1-t \right)^r (1-tz)^{k-r-1} dt, $\\

\item $ \deg P_r=\deg Q_r=r \quad \text{and} \quad \deg E_r=k-r-1, $\\

\item $P_r(z)-(1-z)^{k} Q_r(z)=(-1)^r z^{2r+1} E_A(z),$\\

\item $P_r(z) Q_{r+1}(z)-Q_r(z)P_{r+1}(z)=c z^{2r+1} $ for some non-zero constant $c$.
\end{enumerate}

\end{lemma}

An important feature of this approximations is the ratio $\frac{r}{k} $. Smaller ratios give better results, but only for Larger values of $\left \lvert \beta \right \rvert$. To get more general results, we take
\begin{equation} \label{k,r value}
k=5j, \quad  r= 4j-g,
\end{equation}   
where $g \in \{0,1 \} $, 
With our choice of $k$ and $r$ we can rewrite  $Q_r(z)$, $P_r(z)$ and $E_r(z)$ of Lemma \ref{second form} as follows: 
$$ P_r(z)= (-1)^{g}\sum_{i=0}^{4j-g} \binom{9j-g}{i} \binom{8j-2 g -i}{4j- g}(-z)^i,$$
$$ Q_r(z)= \sum_{i=0}^{4j-g} \binom{8j-2g-i}{4j-g} \binom{j+g-1+i}{i}(z)^i,$$
$$ E_r(z)= \sum_{i=0}^{j+g} \binom{4j-g+i}{i} \binom{9j-g}{8j-2g+i+1} (-z)^i .$$
As usual to use Pad\'{e} approximations, we need explicit bounds on $Q_r(z_0)$ and $E_r(z_0)$.  As noted before $P_r$, $Q_r$ and $E_r$ are all polynomials with integer coefficients. By dividing  these polynomials by the GCD of their coefficients, we get polynomials with integral coefficients and smaller heights and consequently, a better approximation. Let's define 
$$c_{g}(j)= \gcd_{i \in \{ 0,1, \cdots,4j-g \}}  \binom{8j-2g-i}{4j-g} \binom{j+g-1+i}{i}. $$
Then $p^{*}_r(z)=c_{g}(j)^{-1} P_r(z),Q_{r}^{*}(z)=c_{}(j)^{-1} Q_r(z) $ and $E_r^{*}(z)=c_{g}(j)^{-1} E_r(z)$ are all polynomials with integer coefficients.Obviously, part 4 and 5 of Lemma \ref{lemma2.1}, namely, 

 $$P_r^{*}(z)-(1-z)^{k} Q_r^{*}(z)=(-1)^r z^{2r+1} E_A^{*}(z)$$
and 
$$ P_r(z)g^{*} Q_{r+1}^{*}(z)-Q_r^{*}(z)P_{r+1}^{*}(z)g=c z^{2r+1}$$hold for the new polynomials for some non-zero constant $c$. Form \cite{BFT2} we have the following result to bound $c_{g}(j)$.
\begin{lemma}
For $j >50$ and $g \in  \{0,1 \}$, we have 
 $$ c_{g}(j) > 2.943^j.$$
\end{lemma}
\proof  This is a special case of Proposition 5.1 in \cite{BFT2} with $d=4$ $c=5$ and $m=j$. 
\subsection{Explicit bounds for approximations}
From now on, we take :
$$ z_0=1-\frac{\bar{\beta}}{\beta}=\frac{\lambda}{\beta},$$
where $\lambda = 2 \sqrt{-D}$ if $p$ is an odd prime, and $\lambda= \sqrt{-D}$ if $p=2$.

 To apply Pad\'{e} approximation to $\left( \frac{\bar{\beta}}{\beta}\right)^k$ we need explicit bounds on $Q_r(z_0)$ and $E_r(z_0)$. 

\subsubsection{Upper bound for $\left \lvert Q_r^{*}(z_0) \right \rvert$  }

For $g=1$ one can get much stronger upper bound for $ Q_r^{*}(z_0)$. Therefore, to determine an upper bound for $\left \lvert Q_r^{*}(z_0) \right \rvert$ we assume $g=0$

First, we recall  a lemma \cite{B} to give an upper bound for $ \frac{(k+r)!}{(k-r-1)! r! r!}.$ 
\begin{lemma}
For positive integers $A$, $B$ and $C$,
$$ \frac{(A+B+C!)}{A! B! C!} < \frac{1}{2 \pi} \sqrt{\frac{A+B+C}{ABC}} \frac{(A+B+C)^{A+B+C}}{A^A B^B C^C}.$$ 
\end{lemma}
From this lemma with suitable choice of $A$, $B$ and $C$ we have 
 $$ \frac{(k+r)!}{(k-r-1)! r! r!}= \frac{(9j)!}{(j-1)! ((4j)!)^2} < \frac{3}{8 \pi} \left( 3^{18} 2^{-16} \right ) ^j .$$

For $0<t <1 $, since we have $1-z_0 = \frac{\bar{\beta}}{\beta} $, we get :
$$ \left \lvert 1-(1-z_0)t\right \rvert ^2 = 1- 2bt+t^2,$$
where $b= 1- \frac{2D}{\left \lvert \beta \right \rvert ^2}$. From condition \eqref{sigmacondition}, since $D > 12 $, we have $b$ is a positive number  between 0.953 and 1. With smaller value of $b$ we get a bigger value for $ \left\lbrace (1-z)^4 z( 1-2bz+z^2)^2 \right \rbrace $ . It means:

$$ \max_{t \in [0,1]} \left\lbrace (1-t)^4 t( 1-2bt+t^2)^2 \right \rbrace \leq 0.044479. $$
Also, under the same condition we get  
$$ \int_{0}^{1} (1-t)^4 ( 1-2bt+t^2)^2 dt< 0.114552. $$
Therefore, 
\begin{equation*}
\begin{split}
  &\left \lvert   \frac{(k+r)!}{(k-r-1)! r! r!}  \int_{0}^{1} t^{k-r-1}  \left(1-t \right)^r (1-t+z_0t)^r dt \right \rvert \\
                   \leq & \frac{3}{8 \pi} \left( 3^{18} 2^{-16} \right ) ^j  \int_{0}^{1} \left((1-t)^4 t ( 1-2bt+t^2)^2 \right)^{j-1} (1-t)^4 ( 1-2bt+t^2)^2 dt  \\ 
                   \leq &  \frac{3}{8 \pi} \left( 3^{18} 2^{-16} \right ) ^j 0.044479^{j-1} \int_{0}^{1} (1-t)^4 ( 1-2bt+t^2)^2 dt \\
                   \leq & 0.308 \times (262.9407)^j.
                   \end{split}
\end{equation*}
  Thus, it follows 
  \begin{equation} \label{Q upper bound}
    \left \lvert Q_{r}^{*} (z_0) \right \rvert < 0.308 \times (89.3445)^j.
   \end{equation}
   
   \subsubsection{An Upper bound for $\left \lvert E_r^{*}(z_0) \right \rvert$ } 
    
  For any $t \in [0,1] $ we have 
  $$ \left \lvert 1-tz_0 \right \rvert ^2 =1 - \frac{\left \lvert \lambda \right \rvert}{\left \lvert \beta \right \rvert} t(1-t) \leq 1.$$Therefore we get 
  \begin{equation*}
  \begin{split}
   \left \lvert E_r(z_0) \right \rvert= \left \lvert \frac{(k+r)!}{(k-r-1)! r! r!} \int_{0}^{1} t^{r} \left(1-t \right)^r  (1-tz_0)^{k-r-1}  dt \right \rvert \\
     \leq  \frac{(k+r)!}{(k-r-1)! r! r!} \int_{0}^1 (1-t)^rt^r dt .
   \end{split}
   \end{equation*} 
   On the other hand, 
   $$\int_{0}^1 (1-t)^rt^r  = \frac{r! r!}{(2r+1)!}. $$
 Therefore,  
 $$  \left \lvert E_r(z_0) \right \rvert= \frac{(k+r)!}{(k-r-1)!(2r+1)!}. $$
 Motivated by the last equation, we prove the following lemma.
 \begin{lemma}
Let $A$ and $B$ be positive integers. Then
 $$ \frac{(A+B)!}{A! B!}< \frac{1}{\sqrt{2 \pi}} \sqrt{\frac{A+B}{AB}} \frac{(A+B)^{A+B}}{A^A B^B}. $$  
 \end{lemma}  
 \begin{proof}
 The result follows from the explicit version of Stirling's formula by considering the following inequality for positive integers $A$ and $B$:
  $$ \frac{1}{12(A+B)}- \frac{1}{12A+1/4} - \frac{1}{12B+1/4} <0. $$
 \end{proof}  
 From this lemma one gets a weaker bound for $\left \lvert E_r(z_0) \right \rvert $ when $g=1$. Indeed, we have 
  $$ \left \lvert E_r(z_0) \right \rvert= \frac{(9j-1)!}{(j)! (8j-1)!} < \frac{0.377}{\sqrt{j}} \left( \frac{9^9}{8^8} \right)^j. $$
  So we have 
  \begin{equation} \label{upper bound for E}
  \left \lvert E_r^{*}(z_0) \right \rvert < \frac{0.377}{j}\times (7.847 )^j.   
 \end{equation}
  To summarize we have:
  \begin{equation} \label{first approximation}
   \beta^k P_{r}^{*}(z_0) - \bar{\beta}^k Q_r^{*}(z_0)= \beta^{k-2r+1} \lambda ^{2r+1}  E_r^{*}(z_0), 
  \end{equation}
   with explicit exponential bounds on  $P_r^{*}(z_0)$ and $ E_r^{*}(z_0)$ in terms of $r$ as desired.

\section{Second approximation - Algebraic set up}  \label{section3}
We assumed that the equation $x^2+D=p^n$ has a solution $(x_0,n_0)$. Let the equation $x^2+D=p^n.m$ has a solution with $n \gg n_0$. Then, working in the ideals of ring of algebraic integers of the number field $\mathbb{Q} \sqrt{-D}$, we will relate this solution to a given solution $(x_0,n_0)$. This enables us to find an approximation for $\left(\frac{\bar{\beta}}{\beta}\right)^k$. 

 \begin{theorem} \label{theoremsection3}
 Consider $x_0^2+D=p^{n_0}$, and $x^2+D=p^n.m$ with $n > 5n_0$. Then there exist a rational integer $j$ and an algebraic integer $\mu$ in the number field $\mathbb{Q} \sqrt{-D}$ such that $\beta^k \mu -\bar{\beta}^k \bar{\mu} = \pm \lambda $, where $k=5j$.
 \end{theorem}
\begin{proof}
\begin{equation} \label{1}
x_0^2+D=p^{n_0}
\end{equation}
and 
\begin{equation} \label{2}
x^2+D=p^{n}.m
\end{equation}
First assume that $p$ is an odd prime number. Since $\left( \frac{-D}{p} \right)=1 $ the ideal $(p)$ of $  \mathbb{Q}\sqrt{-D}$ splits into two prime ideals $(\alpha)$ and $(\alpha')$ with $(p)=(\alpha)  ({\alpha'})$.
Factoring equation \eqref{1}, in the ideals of integers of the number field $\mathbb{Q} \sqrt{-D}$ we get :
\begin{equation} \label{3}
\left( x_0+ \sqrt{-D} \right)  \left( x_0- \sqrt{-D} \right)= (\alpha)^{n_0}  (\alpha')^{n_0},
\end{equation}
where $\alpha$ and ${\alpha'}$ are prime ideals, therefore they both divide at least one of the two factors on the left side of \eqref{3}. Moreover, since none of the factors on the left side of \eqref{3} belong to ideal $(p)$, the ideals $(\alpha)$ and $(\alpha')$ can not both divide one of the factors simultaneously. Therefore we can assume $$\left( x_0+ \sqrt{-D} \right)  = (\alpha)^{n_0} \quad \left( x_0- \sqrt{-D} \right) =  (\alpha')^{n_0}.$$ Both  ideals $(\alpha)^{n_0} $ and $(\alpha')^{n_0} $ are  principal ideals. Define $\beta=x_0+ \sqrt{-D}$ as a generator of the ideal $(\alpha)^{n_0}$ and  $\beta'=x_0- \sqrt{-D}$ as a generator of the ideal $(\alpha')^{n_0}$. Factoring the equation \eqref{2} in the ideals of number field $\mathbb{Q} \sqrt{-D}$ we have :
$$\left( x+ \sqrt{-D} \right)  \left( x- \sqrt{-D} \right)= (\alpha)^{n}  (\alpha')^{n} (m).$$
By exactly similar argument, each one of the factors $\left( x \pm \sqrt{-D} \right)$ belongs to exactly one of the ideals $(\alpha)^{n}$ and $(\alpha')^{n} $. Let's assume $\left( x+ \sqrt{-D} \right)$ belongs to the ideal $(\alpha)^{n}$. Then it belong to any ideal of $(\alpha)^i$ with $i<n$. Since $n >5n_0$, there exists an integer $j$ with $k=5j$ such that $n=5n_0j+l=n_0k+l$ where, $ l < 5n_0$. Therefore, $\left( x+ \sqrt{-D} \right)$ belongs to the ideal $\left(\alpha^{n_O} \right)^k= \left(\beta \right)^k$. On the other hand, $\left(\beta \right)^k$ is a principal ideal with $\beta^k$ as a generator. Hence,  there is an algebraic integer $\mu$ in the number field $\mathbb{Q} \sqrt{-D}$ such that $x+\sqrt{-D}= \beta^k \mu $. Taking conjugates we have $x-\sqrt{-D}= \bar{\beta}^k \bar{\mu} $. If $\left( x+ \sqrt{-D} \right)$ belongs to the ideal $(\alpha')^{n}$, the same steps show that there is an algebraic integer $\mu$ in the number field $\mathbb{Q} \sqrt{-D}$ such that $x+\sqrt{-D}= \bar{\beta}^k \mu $. Thus, in any case, $\beta^k \mu -\bar{\beta}^k \bar{\mu} = \pm 2 \sqrt{D} $.
For $p=2$ the argument is similar with minor modifications. Let $x_0^2+D=2^{n_0}.$ Since $-D \equiv 1\pmod 4$, we can  factorize the equation as
$$ \frac {\left( x_0+ \sqrt{-D} \right)}{2} \frac{\left( x_0- \sqrt{-D} \right)}{2} = 2^{n-2}. $$ 
There is an algebraic integer $\mu$ in the number field $\mathbb{Q} \sqrt{-D}$ where
either
$$ \frac{x+ \sqrt{-D}}{2}= \beta^k \mu \quad \text{and} \quad \frac{x- \sqrt{-D}}{2}=\bar{\beta}^k \bar{\mu}, $$   
or
$$ \frac{x+ \sqrt{-D}}{2}= \bar{\beta}^k \mu \quad \text{and} \quad \frac{x- \sqrt{-D}}{2}={\beta}^k \bar{\mu}. $$
In any case $\beta^k \mu -\bar{\beta}^k \bar{\mu} = \pm  \sqrt{D} = \pm \lambda $.
\end{proof}
What this theorem states is that, whenever $\left \lvert \mu \right \rvert$ is small compare to  $\beta^k$, there exists a good approximation for $ \left( \frac{\bar{\beta}}{\beta} \right)^k$. Note that $\mu \bar{\mu}= 2^l.m$, where $l <5n_0$ for the odd prime $p$ and $l < 5(n_0-2)$ for $2$. Which means, whenever $m$ is small compare to $\beta^k$ there exists a good approximation for $ \left( \frac{\bar{\beta}}{\beta} \right)^k$. We also like to mention the following inequalities that will be helpful in the proof of Theorem \ref{theorem1}. If $P=2$ then $\left\lvert \beta^k \mu \right \rvert =\frac{\sqrt{x^2+D}}{2} > 0.7x$, and if $p$ is an odd prime, then $\left\lvert \beta^k \mu \right \rvert ={\sqrt{x^2+D}}>x$   

\section{proof of theorem \ref{theorem1}} \label{section4}

In this section, using approximations of $\left(\frac{\bar{\beta}}{\beta}\right)^k$ in equations \eqref{first approximation}, and Theorem \ref{theoremsection3}, we prove theorem \ref{theorem1}. Through the proof we assume that 
$$ k=5j, \quad  j \geq 50 \quad n> M \geq 250n_0. $$
Let $(x,n,m)$ be a solution to the equation $x^2+D=p^n.m$, with $n > M$. Then $x^2+D \leq p^{M+1}$. Since $x_0^2+D=p^{n_0}$, we can conclude that $x > p^{125n_0}$.
If we multiply both sides of the equation \eqref{first approximation} by $\beta^r$, we get
\begin{equation} {\label{final approximation}}
\beta^k P - \bar{\beta}^k Q =E, 
\end{equation}
where 
$$P=\beta^r P_r^{*}(z_0), \quad   Q=\beta^r Q_r^{*}(z_0),  \quad \text{ and  }   E=\beta^{k-r-1} \lambda^{2r+1} E_{r}^{*}(z_0).$$
Note that by considering the degree of $P_r$, $Q_r$ and $E_r$, we can see that $P$, $Q$ and $E$ are algebraic integers in the quadratic number field $\mathbb{Q}\sqrt{-D}$.\\ 
Form equations \eqref{final approximation} and \eqref{theoremsection3} we obtain
$$ \beta^{k} \left(Q \mu - P \bar{\mu} \right) = \pm Q \lambda - E \bar{\mu}.$$ 
From part 5 of Lemma \ref{lemma2.1}, for at least one of the values of $g $ in $r=4j-g$, the left hand side of the equation above is nonzero. Thus, $Q \mu - P \bar{\mu}$ is an algebraic integer in the number field $\mathbb{Q}(\sqrt{-D})$, so 
its norm is bigger than 1 . Therefore, we get 
$$ \left \lvert \beta^{k} \right \rvert \leq \left \lvert Q \right \rvert \left \lvert \lambda \right \rvert + \left \lvert E \right \rvert \left \lvert \bar{\mu} \right \rvert. $$From \eqref{Q upper bound} and conditions \ref{sigmacondition} we have
$$\left \lvert Q \right \rvert \left \lvert \lambda \right \rvert  < 0.238074 \times (89.3445)^j  \beta^{r} \beta ^{0.4873}.$$
So whenever $\beta > 90.93$ we have 
 $$    \left \lvert Q \right \rvert \left \lvert \lambda \right \rvert <  \frac{9}{10} \left \lvert \beta^{k} \right \rvert. $$
Hence, 
 $$ \frac{1}{10} \left \lvert \beta \right \rvert ^k \leq \left \lvert \beta \right \rvert ^{k-r-1} \left \lvert \lambda \right \rvert ^{2r+1} \left \lvert E_r^{*}(z_0) \right \rvert  \left \lvert \mu \right \rvert. $$ 
It follows that
$$ \left \lvert \mu \right \rvert \geq  \left ( \frac{\beta^{r+1}}{\lambda^{2r+1}} \right) \frac{1}{10 E_r^{*}(Z_0)} > \left( \frac{\beta^4}{7.847 (\lambda)^8 } \right) ^j \frac{\lambda \sqrt{j}}{3.7}.$$If conditions \eqref{sigmacondition} is satisfied we have 

  $$\left( \left \lvert \beta \right \rvert  ^k \right)^{\frac{\sigma+0.04}{1.96-\sigma}}= \left( \left \lvert \beta \right \rvert ^{\frac{5\sigma+0.2}{1.96-\sigma}} \right)^j <   \left( \frac{\beta^4}{\lambda^8 7.847} \right)^j \leq \frac{\left \lvert \mu \right \rvert}{6.89}. $$It is easy to check that $\beta ^k \mu > 0.7 x $. Therefore, $(0.7x)^{\frac{\sigma+0.04}{1.96-\sigma}} < (\beta ^k \mu)^{\frac{\sigma+0.04}{1.96-\sigma}} $. But from the above we have 
$\left( \left \lvert \beta \right \rvert  ^k \right)^{\frac{\sigma+0.04}{1.96-\sigma}} < \frac{1}{6.89}\left \lvert \mu \right \rvert $, so we get

 $$ \mu ^{\frac{2}{1.96-\sigma}} > 5.24 x^{\frac{\sigma+0.04}{1.96-\sigma}} \Longrightarrow \mu^2 > 6.32 x^{\sigma+0.04}. $$  
 Finally since $x> p^{125n_0}$ we have, 
 $$ m \geq \frac{\left \lvert \mu \right \rvert ^2}{p^{5n_0-1}}> \frac{6.32 x^{\sigma+0.04}}{p^{5n_0-1}} > x^{\sigma}. $$
This completes proof of Theorem \ref{theorem1}.\\
We would like to add a remark here as we mentioned before. Smaller values of $\frac{r}{k}$  give stronger results for larger values of $\left \lvert \beta \right \rvert $. 
 As an example, we mention the result for the case in which $k=7j$ and $r=6j-\sigma$ without a detailed proof.
 In this case for $\left \lvert \beta \right \rvert >1300$, it is enough to take $\eta_{\epsilon} = \frac{11.76-6 \sigma}{11.48-13\sigma}$, $\ C_{\sigma,p}=(42106)^{\frac{1.96-\sigma}{11.48-13\sigma}}$ and $ C_{\sigma,2}=(10.28)^{\frac{1.96-\sigma}{11.48-13\sigma}}$ to obtain similar result as Theorem \ref{theorem1}.
 
 \section{the equation $x^2+76=101^n.m$} \label{finalsection} 
As an application of Theorem \ref{theorem1}, we consider the equation $x^2+76=101^n.m$. The value $M$ in corollary  \ref{corollary1} might look huge. In practice it is easy to check the values less than $P^{M}$. To see this we present the following result.
 \begin{theorem}
Let $x^2+76=101^n.m $. Then either $x \in \{5,1015\}$ or $m> x^{0.14} $. 
 \end{theorem}
 \begin{proof}
  The equation $1015^2+76=101^3$ gives a solution for Ramanujan-Nagell equation $x^2+76=101^n$. This can be considered as a huge solution. To be more precise, $(1015,3)$ is a huge solution for any $\sigma $ with $\sigma \leq 0.14$. Therefore, from Corollary \ref{corollary1} the theorem holds for $x > 101^{750}$.  
  
 Assume that $x< 101^{750}$ and $m < x^{0.14} $ . Therefore we have to check the equation $x^2+76=101^n.m$ for the values $n \leq 750$ and $x < 101^n$. For any $n<750$ we solve the equation $x^2+76 \equiv 0 \pmod {101^n} $ and find $m= \frac{x^2+76}{101^n}$ . From Hensel's lemma, for each $n$ there are two values that need to be checked which can be easily achieved using a recursive relation. In this way it is easy to confirm the theorem  for all the values $n<750$ and actually for $x< 101^{3000}$ we have a stronger result $m > x^{0.9}$ or $x \in \{5, 1015 \} $.  
 \end{proof}

    \end{document}